\documentclass[12pt]{article}

\usepackage{amsfonts,amssymb,latexsym,amsmath}

\usepackage[english]{babel}
\usepackage[margin=1in]{geometry}

\usepackage{enumitem}
\usepackage{tikz}
\tikzset{font={\fontsize{8pt}{12}\selectfont}}
\usepackage{verbatim}
\usetikzlibrary{matrix}

\newtheorem{thm}{Theorem}[section]
\newtheorem{prop}[thm]{Proposition}
\newtheorem{lem}[thm]{Lemma}
\newtheorem{cor}[thm]{Corollary}

\newenvironment{proof}
{\par\addvspace{0.3cm}\noindent{\rm Proof. }}
{\nopagebreak\mbox{}\hfill $\Box$\par\addvspace{0.25cm}}

\newcommand{\T}{\mathbb{T}}
\newcommand{\Z}{\mathbb{Z}}

\newcommand{\C}{\mathbb{C}}

\newcommand{\kA}{{\frak A}}
\newcommand{\kB}{{\frak B}}
\newcommand{\kD}{{\frak D}}
\newcommand{\kJ}{{\frak J}}

\newcommand{\tQC}{\widetilde{QC}}
\newcommand{\tC}{\widetilde{C}}

\newcommand{\om}{\omega}

\newcommand{\eps}{\varepsilon}
\newcommand{\iy}{\infty}

\DeclareMathOperator{\closid}{clos\, id}
\DeclareMathOperator{\clos}{clos}

\numberwithin{equation}{section}

\begin{document}

\title{On the maximal ideal space of even quasicontinuous functions on the unit circle}
\author{Torsten Ehrhardt\thanks{tehrhard@ucsc.edu}\\
   Department of Mathematics\\
        University of California\\
       Santa Cruz, CA 95064, USA
  \and   Zheng Zhou\thanks{zzho18@ucsc.edu}\\
 Department of Mathematics\\
        University of California\\
       Santa Cruz, CA 95064, USA}

\date{}

\maketitle

\begin{abstract}
Let $PQC$ stand for the set of all piecewise quasicontionus function on the unit circle, i.e., 
the smallest closed subalgebra of $L^\infty(\T)$ which contains the classes of  all piecewise continuous function $PC$ and all
quasicontinuous functions $QC=(C+H^\infty)\cap(C+\overline{H^\infty})$. We analyze the fibers of the maximal ideal spaces
$M(PQC)$ and $M(QC)$ over maximal ideals from $M(\widetilde{QC})$, where $\widetilde{QC}$ stands for the $C^*$-algebra of all
even quasicontinous functions. The maximal ideal space $M(\tQC)$ is decribed and partitioned into
various subsets  corresponding to different descriptions of the fibers.
\end{abstract}



\section{Introduction}\label{s:1}

Let $L^\infty(\T)$ stand for the $C^*$-algebra of all (complex-valued) Lebesgue measurable and essentially bounded functions on the unit circle
$\T=\{\,t\in\C\,:\,|t|=1\,\}$, let $C(\T)$ stand for the class of all continuous functions on $\T$, and let $PC$ stand for the set of all
piecewise continuous functions on $\T$, i.e.,  all functions $f:\T\to\C$ such that the one-sided limits
$f(\tau\pm0)=\lim\limits_{\varepsilon\to +0} f(\tau e^{\pm i \varepsilon })$ exist at each $\tau\in\T$. The class of quasicontinuous functions is 
defined by
$$
QC=(C+H^\infty)\cap (C+\overline{H^{\infty}}),
$$
where $H^\iy$ stands for the Hardy space consisting of all $f\in  L^\iy(\T)$ such that its Fourier coefficients 
$f_n=\frac{1}{2\pi}\int_0^{2\pi} f(e^{ix})e^{-inx}\, dx$ vanish for all $n<0$. The space $\overline{H^\iy}$ is the Hardy space of 
all functions $f\in L^\infty(\T)$ such that $f_n=0$ for all $n>0$.

The Toeplitz and Hankel operators  $T(a)$ and $H(a)$ with $a\in L^\iy(\T)$ acting on $\ell^2(\Z_+)$ are defined by the infinite matrices
$$
T(a) =(a_{j-k})_{j,k=0}^\iy,\qquad H(a)=(a_{j+k+1})_{j,k=0}^\iy\,.
$$
Quasicontinuous functions arise in connection with Hankel operators. Indeed, it is known
that both $H(a)$ and $H(\tilde{a})$ are compact if and only if $a\in QC$ (see, e.g., \cite[Theorem 2.54]{BS}).
Here, and what follows, $\tilde{a}(t):=a(t^{-1})$, $t\in \T$. 

\smallskip 
Sarason \cite{sarason1977}, generalizing earlier work of Gohberg/Krupnik \cite{GohKru69} and Douglas \cite{douglas1972banach},  established necessary and sufficient conditions
for Toeplitz operators $T(a)$ with $a\in PQC$ to be Fredholm.  This result is based on two ingredients.
Firstly, due to Widom's formula $T(ab)=T(a)T(b)+H(a)H(\tilde{b})$,  Toeplitz operators $T(a)$ with $a\in QC$ commute with  other Toeplitz operators $T(b)$, $b\in L^\iy(\T)$, modulo compact operators.
Hence $C^*$-algebras generated by Toeplitz operators can be localized over $QC$.
Secondly, in case of the $C^*$-algebra generated by Toeplitz operators $T(a)$ with $a\in PQC$, the local quotient algebras arising from the localization allow an explicit description, which is facilitated by the characterization of the fibers of the maximal ideal space $M(PQC)$ over maximal ideals   $\xi \in M(QC)$. These underlying results were also developed by Sarason \cite{sarason1975functions,sarason1977},
and we are going to recall them in what follows.

\smallskip
Let $\kA$ be a commutative  $C^*$-algebra, and let $\kB$ be a $C^*$-subalgebra such that both contain the same unit element.
Then there is a natural continuous map between the maximal ideal spaces,
$$
\pi: M(\kA) \to M(\kB), \qquad \alpha \mapsto \alpha|_{\kB}
$$
defined via the restriction. For $\beta\in M(\kB)$ introduce
$$
M_\beta(\kA)= \{\; \alpha\in M(\kA)\;:\; \alpha|_{\kB}=\beta\,\} = \pi^{-1}(\beta),
$$
which is called the {\em fiber of $M(\kA)$ over $\beta$}. The fibers $M_\beta(\kA)$ are compact subsets of $M(\kA)$, 
and $M(\kA)$ is the disjoint union of all $M_\beta(\kA)$.
Because $\kA$ and $\kB$ are $C^*$-algebras,  $\pi$ is surjective, and therefore each fiber $M_\beta(\kA)$ is non-empty
(see, e.g.,  \cite[Sect.~1.27]{BS}).

Corresponding to the embeddings between the $C^*$-algebras $C(\T)$, $QC$, $PC$, and $PQC$, which are depicted in first diagram below,
there are natural maps between the maximal ideal spaces shown in the second diagram:
\begin{center}
\begin{tikzpicture}[thick,scale=1]
 \matrix (m) [matrix of math nodes,row sep=2em,column sep=4em,minimum width=3em]
  {
     PQC & {QC}&   M(PQC)  & M({QC})\\
     PC & {C}(\T)  &  M(PC)\cong {\T\times\{+1,-1\} }& M({C(\T)})\cong \T \\};
  \path[-stealth]
        (m-1-2) edge node [above] {}  (m-1-1)
        (m-2-2) edge node [above] {} (m-2-1)
        (m-2-1) edge node [left]  {} (m-1-1)
        (m-2-2) edge node [left]  {} (m-1-2)
        (m-1-3) edge node [above] 
        					{}  (m-1-4)
        (m-1-3) edge node [below]  {} (m-2-3)
        (m-2-3) edge node [above]  
        					{}  (m-2-4)
        (m-1-4) edge node [below]  {} (m-2-4);
\end{tikzpicture}
\end{center}
Therein the identification of $y\in M(PC)$ with $(\tau,\sigma)\in\T\times \{+1,-1\}$ is made through
$y(f)=f(\tau \pm 0)$ for $\sigma=\pm1$, $f\in PC$. 

Let $M_\tau(QC)$ stand for the fiber of $M(QC)$ over $\tau\in\T$, i.e., 
$$
  M_\tau(QC) = \{\,\xi\in M(QC): \xi(f) = f(\tau) \mbox{ for all } f\in C(\mathbb{T})\,\},
$$       			 
and define
$$
  M_\tau^\pm(QC)=\Big\{\, \xi\in M(QC)\,:\, \xi(f)=0 \mbox{ whenever }
\limsup_{t\to \tau\pm 0} |f(t)|=0 \mbox{ and } f\in QC\;\Big\}.
$$
Both $M_\tau^+(QC)$ and $M_\tau^-(QC)$ are closed subsets of $M_\tau(QC)$. Sarason introduced another subset $M_\tau^0(QC)$ (to be defined in \eqref{Mtau0} below) and established the following result (see \cite{sarason1977}, or  \cite[Prop.~3.34]{BS}).

\begin{prop}
\label{prop52}
Let $\tau\in \mathbb{T}$. Then
\begin{equation}
        M_\tau^0(QC) = M_\tau^+(QC)\cap M_\tau^-(QC), \quad  M_\tau^+(QC)\cup M_\tau^-(QC) = M_\tau(QC).
\end{equation}
\end{prop}

The previous definitions and observations are necessary to analyze the fibers of $M(PQC)$ over $\xi\in M(QC)$.
In view of the  second  diagram above, for given $z\in M(PQC)$ we can define the restrictions
$\xi=z|_{QC}$, $z|_{C(\T)}\cong \tau \in \T$, and $y=z|_{PC}\cong(\tau,\sigma)\in\T\times \{+1,-1\}$.
Note that $\xi\in M_\tau(QC)$.  Consequently, one has a natural map
\begin{equation}\label{z-xi-si}
    z \in M(PQC)    \mapsto    (\xi,\sigma)\in M(QC)\times \{+1,-1\}.
\end{equation}
This map is injective because $PQC$ is generated by $PC$ and $QC$. Therefore, $M(PQC)$ can be identified
with a subset of $M(QC)\times \{+1,-1\}$. With this identification, the fibers 
$M_\xi(PQC)=\{\, z\in M(PQC)\,:\, z|_{QC}=\xi\, \}$ are given as follows (see \cite{sarason1975functions}, or \cite[Thm.~3.36]{BS}).

\begin{thm}
\label{thm53}
Let $\xi\in M_\tau(QC)$, $\tau\in \T$. Then
\begin{itemize}
\item[(a)]
$M_\xi(PQC)=\{\, (\xi,+1)\,\}$ for $\xi\in M_\tau^+(QC)\setminus M_\tau^0(QC)$;
\item[(b)]
$M_\xi(PQC)=\{\, (\xi,-1)\,\}$ for $\xi\in M_\tau^-(QC)\setminus M_\tau^0(QC)$;
\item[(c)]
$M_\xi(PQC)=\{\, (\xi,+1),(\xi,-1)\,\}$ for $\xi\in M_\tau^0(QC)$.
\end{itemize}
\end{thm}

In order to describe the content of this paper, let us consider what happens if one wants to develop
a Fredholm theory for operators from the $C^*$-algebra generated by Toeplitz and Hankel operators with $PQC$-symbols \cite{Si87}.
In this situation, one cannot use localization over $QC$
because the commutativity property fails. However, one can localize over 
$$\tQC=\{\,a\in QC\,:\, a=\tilde{a}\,\},$$ 
the $C^*$-algebra of all even quasicontinuous functions. Indeed, due to the identity $H(ab)=T(a)H(b)+H(a) T(\tilde{b})$, any
$T(a)$ with $a\in \tQC$ commutes with any $H(b)$, $b\in L^\iy(\T)$, modulo compact operators.
When faced with the problem of identifying the local quotient algebras, it is necessary to 
understand the fibers of $M(PQC)$ over $\eta\in M(\tQC)$. This is what this paper is about.

\smallskip 
When $\tQC$ and  the $C^*$-algebra $\tC(\T)$ of all even continuous functions are added to the picture,
one arrives at the following diagrams:
\begin{center}
\begin{tikzpicture}[thick,scale=1]
 \matrix (m) [matrix of math nodes,row sep=2em,column sep=4em,minimum width=3em]
  {
     PQC & QC & \tQC& M(PQC) &  M(QC)  & M(\widetilde{QC})\\
     PC & C(\T) & \widetilde{C}(\T)  & M(PC)&  M(C)\cong {\T }& M(\widetilde{C})\cong \overline{\T_+} \\};
  \path[-stealth]
        (m-1-2) edge node [above] {}  (m-1-1)
        (m-1-3) edge node [above] {} (m-1-2)
        (m-2-1) edge node [left]  {} (m-1-1)
        (m-2-2) edge node [left]  {} (m-1-2)
        (m-2-3) edge node [left]  {} (m-1-3)
        (m-2-2) edge node [above] {} (m-2-1)
        (m-2-3) edge node [above] {} (m-2-2) 
        (m-1-4) edge node [above] {}  (m-1-5)
        (m-1-5) edge node [above] {\footnotesize{$\Psi$}} (m-1-6)
        (m-1-4) edge node [left]  {} (m-2-4)
        (m-1-5) edge node [left]  {} (m-2-5)
        (m-1-6) edge node [left]  {} (m-2-6)        
        (m-2-4) edge node [above] {}  (m-2-5)
        (m-2-5) edge node [above] {\footnotesize{$\Psi'$}} (m-2-6);
\end{tikzpicture}
\end{center}
As before, the diagram on the left shows the embeddings of the $C^*$-algebras, and the one on the right 
displays the corresponding (surjective) mappings between the maximal ideal spaces. Here $\T_+=\{\,t\in\T\,:\, \mathrm{Im}(t)>0\,\}$
and $\overline{\T_+}=\T_+\cup\{+1,-1\}$. The map $\Psi'$ is defined in such a way that the pre-image
of $\tau\in\overline{ \T_+}$ equals the set $\{\tau,\overline{\tau}\}$, which consists of either one or two points.

Recall that Theorem \ref{thm53} describes the fibers of $M(PQC)$ over $\xi \in M(QC)$. Hence if we want to 
understand the fibers of $M(PQC)$ over $\eta\in M(\tQC)$,  
it is sufficient to analyze the fibers of $M(QC)$ over $\eta\in M(\tQC)$.
Let
\begin{equation}\label{xi-hat}
        \Psi: M(QC)\to M(\widetilde{QC}), \enskip\xi\mapsto \hat{\xi}:=\xi|_{\tQC}\,,
\end{equation}
be the (surjective) map shown in the previous diagram. 
For $\eta\in M(\widetilde{QC})$ define
        \begin{equation}
        M^\eta(QC) = \{\,\xi\in M(QC): \hat{\xi} = \eta\,\},
        \end{equation}
the fiber of $M(QC)$ over $\eta$.  
Let us also define the fibers of $M(\tQC)$ over $\tau\in \overline{\T_+}$,
\begin{equation}\label{f.tau.tQC}
         M_\tau(\widetilde{QC}) = \{\,\eta\in M(\widetilde{QC}): \eta(f) = f(\tau) \mbox{ for all } f\in \widetilde{C}(\mathbb{T}) \,\}.
\end{equation}
Notice that we have the disjoint unions
\begin{equation}\label{unions}
    M(QC)=\bigcup_{\eta\in M(\widetilde{QC})} M^\eta(QC),\quad
    M(QC) =\bigcup_{\tau\in\T} M_{\tau}(QC),\quad
    M(\tQC)=\bigcup_{\tau\in\overline{\T_+}} M_\tau (\tQC).
\end{equation}
Furthermore, it is easy to see that  $\Psi$ maps 
\begin{equation}\label{eq1.7}
    M_\tau(QC)\cup M_{\bar{\tau}}(QC) 
\end{equation}
onto $M_\tau(\tQC)$ for each $\tau\in \overline{\T_+}$.

\smallskip
The main results of this paper concern the description of the fibers $M^\eta(QC)$ and the decomposition of
$M_\tau(\tQC)$ into disjoint sets, analogous to the decomposition of $M_\tau(QC)$
into the disjoint union of
\begin{equation}\label{f1.8}
   M^0_\tau(QC), \quad M^+_\tau(QC)\setminus M^0_\tau(QC), \quad \mbox{and}\quad 
   M^-_\tau(QC)\setminus M^0_\tau(QC)
\end{equation}
(see Proposition \ref{prop52}). This will be done in Section 3. In Section 2 we establish auxilliary results.
In Section 4 we decribe the fibers $M^\eta(PQC)$ of  $M(PQC)$ over $\eta\in M(\tQC)$.

\smallskip
Some aspects of the relationship between $M(QC)$ and $M(\tQC)$ were already mentioned by Power \cite{Pow80}.
They were used by Silbermann \cite{Si87} to established a Fredholm theory for operators from the $C^*$-algebra 
generated by Toeplitz and Hankel operators with $PQC$-symbols. 
Our motivation for presenting the results of this paper comes from the goal
of establishing a Fredholm theory and a stability theory for the finite section method 
for operators taken from the $C^*$-algebra generated by the singular integral operator on $\T$, the flip operator, and multiplication operators
by (operator-valued) $PQC$-functions \cite{EZ}. This generalizes previous work \cite{ehrhardt1996symbol,ehrhardt1996finite} and requires the results established here.


\section{Approximate identities and VMO}
\label{sec:2}

In order to examine the relationship between $M(QC)$ and $M(\widetilde{QC})$,  we need to recall some results and definitions concerning 
$QC$ and $M(QC)$. 
For $ \tau = e^{i\theta}\in \mathbb{T}$ and $ \lambda\in\Lambda:=[1,\infty)$ let us define the moving average,
\begin{equation}\label{mlambda}
        (m_\lambda a)(\tau) = \frac{\lambda}{2\pi}\int_{\theta-\pi/\lambda}^{\theta+\pi/\lambda}a(e^{ix})\,dx.
 \end{equation}
Since each pair $(\lambda,\tau)\in \Lambda\times \T$ induces a bounded linear functional  $\delta_{\lambda,\tau}\in QC^*$,
 \begin{equation}
        \delta_{\lambda,\tau}:QC\to \mathbb{C},  \quad a\mapsto (m_\lambda a)(\tau),\label{deltam}
 \end{equation}
the set $\Lambda\times \T$ can be identified with a subset of $QC^*$. 
In fact, we have the following result, where we consider the dual space $QC^*$ with the weak-$^*$ topology (see \cite[Prop.~3.29]{BS}).

\begin{prop}\label{p2.1}
$M(QC)= (\clos_{\,QC^*}(\Lambda\times \mathbb{T}))\setminus (\Lambda \times \mathbb{T}).$
\end{prop}

For $\tau\in\T$, let $M_\tau^0(QC)$ denote the points in $M(QC)$ that lie in the weak-$^*$ closure of $\Lambda\times \{\tau\}$ regarded as a subset of $QC^*$, 
\begin{equation}  
        M_\tau^0(QC) = M(QC)\cap  \clos_{\,QC^*}(\Lambda \times \{\tau\}) .\label{Mtau0}
\end{equation}
Obviously,  $M_\tau^0(QC)$ is a compact subset of the fiber $M_\tau(QC)$. 
We remark that here and in the above proposition one can use arbitrary  approximate identities (in the sense of Section 3.14 in \cite{BS})
instead of the moving average (see \cite[Lemma~3.31]{BS}).

\medskip
For $a\in L^1(\T)$ and $\tau = e^{i\theta}\in \mathbb{T}$, the \emph{integral gap} $\gamma_\tau(a)$ of $a$ at $\tau$ is defined by
\begin{equation}
        \gamma_\tau(a) :=\limsup_{\delta \to +0} 
        \left|\frac{1}{\delta}\int_\theta^{\theta+\delta}a(e^{ix})\,dx - 
        \frac{1}{\delta}\int_{\theta-\delta}^\theta a(e^{ix})\,dx\right|.
\end{equation}
It is well-known \cite{sarason1975functions} that $QC = VMO\cap L^\infty(\T)$, where $VMO\subset L^1(\T)$ refers to the class
of all {\em functions with vanishing mean oscillation} on the unit circle $\T$.  We will not recall its definition here, but refer to  \cite{sarason1975functions,sarason1977,BS}.
In the following lemma (see  \cite{sarason1977} or \cite[Lemma~3.33]{BS}), $VMO(I)$ stands for the class of functions with vanishing mean oscillation on an open subarc $I$ of $\T$. Furthermore, we identify a function $q\in QC$ with its Gelfand transform, a continuous function on $M(QC)$.

\begin{lem}\label{lem53}\hfill
\begin{enumerate}[label=(\alph*)]
\item If $q\in VMO$, then $\gamma_\tau(q) = 0$ for each $\tau\in \mathbb{T}$.
\item If $q\in VMO(a,\tau)\cap VMO(\tau,b)$ and $\gamma_\tau(q)=0$, then $q\in VMO(a,b)$.
\item If $q\in QC$ such that  $q|_{M_\tau^0(QC)} = 0$ and if $p\in PC$, then $\gamma_\tau(pq)=0$.
\end{enumerate}
\end{lem}

Let $\chi_+$ (resp., $\chi_-$) be the characteristic function of the upper (resp., lower) semi-circle. 
The next lemma is based on the preceeding lemma.

\begin{lem}\label{cor54}
Let $q\in QC$.
\begin{enumerate}[label=(\alph*)]
\item If $q$ is an odd function, i.e., $q(t) = -q(1/t)$, then $q|_{M_{1}^0(QC)} = 0$ and $q|_{M_{-1}^0(QC)} = 0$.
\item If $q|_{M_{\pm1}^0(QC)} = 0$, then $pq\in QC$ whenever $p\in PC\cap C(\mathbb{T}\setminus\{\pm1\})$.
\item If $q|_{M_{1}^0(QC)}=0$ and $q|_{M_{-1}^0(QC)}=0$, then $q\chi_+,q\chi_-\in QC$.
\end{enumerate}
\end{lem}
\begin{proof}
For part (a), since $q\in QC$ is an odd function,  it follows  from \eqref{mlambda} that 
$$ 
  \delta_{\lambda,\pm 1}(q)=        (m_{\lambda}q)(\pm 1) = 0 
  \quad\mbox{ for all } \lambda\ge1.
$$  
Therefore, by \eqref{deltam} and \eqref{Mtau0}, $q$ vanishes on $\Lambda\times\{\pm 1\}\subseteq QC^*$ and hence on its closure, in particular, also on $M_{\pm 1}^0(QC)$. 

For part (b) assume that $q|_{M_{\pm 1}^0(QC)}=0$. We use the fact that  $QC = VMO\cap L^\infty$. 
It follows from the definition of $VMO$-functions that the product of a $VMO$-function with a uniformly continuous function
is again $VMO$. Therefore, $pq$ is $VMO$ on the interval $\T\setminus \{\pm 1\}$. By Lemma  \ref{lem53}(c), the integral gap $\gamma_{\pm 1}(pq)$ is zero. Hence $pq$ is $VMO$ on all of $\T$  by Lemma  \ref{lem53}(b). This implies $pq\in QC$. 

For case (c) decompose $q=qc_1+qc_{-1}$ such that $c_{\pm1}\in C(\T)$ vanishes identically in a neighborhood of $\pm1$.  
Then apply the result of (b).
\end{proof}

We will also need the following lemma.

\begin{lem}\label{lem57} 
$\delta_{\lambda, \tau}$ is not multiplicative over $\widetilde{QC}$ for each fixed $\lambda\in [1,\infty)$ and $\tau \in \mathbb{T}$.
\end{lem}
\begin{proof}
Let $\tau=e^{i\theta}$ and consider $\phi(e^{ix})  = e^{ikx} + e^{-ikx}$ with $k\in \mathbb{N}$. Apparently, $\phi\in \widetilde{QC}$. Note that the moving average  is generated by the function
$$
K(x) = \frac{1}{2\pi}\chi_{(-\pi, \pi)}(x), \qquad \delta_{\lambda,\tau}(q)=(m_\lambda q)(e^{i\theta}) 
   =   \int_{-\iy}^\iy  \lambda K(\lambda x) q(e^{i(\theta-x)})\, dx.
$$
Hence, by formula 3.14(3.5) in \cite{BS}, or by direct computation,
 \begin{align*}
        \lefteqn{\delta_{\lambda,\tau}(\phi^2) - \delta_{\lambda,\tau}(\phi)\delta_{\lambda,\tau}(\phi)}
        \hspace*{4ex}\\
        &=  
        (\hat{K}(2k/\lambda)e^{2ki\theta}+ \hat{K}(-2k/\lambda)e^{-2ki\theta} +2) 
        - (\hat{K}(k/\lambda)e^{ki\theta}+ \hat{K}(-k/\lambda)e^{-ki\theta})^2
        \\
        &= 
        2\cos(2k\theta)\left(\frac{\sin(2k\pi/\lambda)}{2k\pi/\lambda} - \left(\frac{\sin(k\pi/\lambda)}{k\pi/\lambda}\right)^2\right) 
        + 2 - 2\left(\frac{\sin(k\pi/\lambda)}{k\pi/\lambda}\right)^2 ,
\end{align*}
where $\hat{K}$ is the Fourier transform of the above $K$. Note that $\frac{\sin x}{x}\to 0$ as $x\to\infty$. Hence, for each fixed $\lambda$, one can choose a sufficiently large $k\in \mathbb{N}$, such that with the corresponding $\phi$,
$$
\delta_{\lambda,\tau}(\phi^2) - \delta_{\lambda,\tau}(\phi)\delta_{\lambda,\tau}(\phi) >1.
$$
Therefore $\delta_{\lambda,\tau}$ is not multiplicative for each $\lambda$ and $\tau$.
\end{proof}


\section{Fibers of $\boldsymbol{M(QC)}$ over $\boldsymbol{M(\tQC)}$}
\label{sec:4}

Now we are going to describe the fibers $M^\eta(QC)$. 
To prepare for it, we make the following definition.
Given $\xi\in M(QC)$, we define its ``conjugate'' $\xi'\in M(QC)$ by
\begin{equation}
  \xi'(q):= \xi(\tilde{q}),\qquad q\in QC.
\end{equation}
Recalling also definition \eqref{xi-hat}, it is clear that $\hat{\xi}=\hat{\xi'}\in M(\widetilde{QC})$. Furthermore, the following statements are obvious:
\begin{enumerate}
\item[(i)]
If $\xi\in M_\tau(QC)$, then $\xi'\in M_{\bar{\tau}}(QC)$.
\item[(ii)]
If $\xi\in M_\tau^\pm(QC)$, then $\xi'\in M_{\bar{\tau}}^\mp(QC)$.
\item[(iii)]
If $\xi\in M_\tau^0(QC)$, then $\xi'\in M_{\bar{\tau}}^0(QC)$.
\end{enumerate}
For the characterization of the fibers $M^\eta(QC)$ we have to distingish whether $\eta\in M_\tau(\tQC)$ with $\tau\in\{+1,-1\}$ or with
$\tau\in \T_+$. In this connection recall the last formula in  \eqref{unions}.


\subsection{Fibers over $\boldsymbol{M_{\tau}(\tQC)}$, $\boldsymbol{\tau\in\{+1,-1\}}$}

For the description of $M^\eta(QC)$ with $\eta\in M_{\pm 1}(\widetilde{QC})$ the following results is crucial.

\begin{prop}\label{p3.1}
If $\xi_1,\xi_2\in M_{\pm 1}^+(QC)$ and $\hat{\xi_1}=\hat{\xi_2}$, then $\xi_1=\xi_2$.
\end{prop}
\begin{proof}
Each $q\in QC$ admits a unique decomposition
$$
q = \frac{q+\tilde{q}}{2} + \frac{q - \tilde{q}}{2} =: q_e + q_o,
$$
where $q_e$ is even and $q_o$ is odd.
By Lemma \ref{cor54}(ac), we have $q_o\chi_-\in QC$, and
\begin{align*}
        \xi_1(q) 
        &= \xi_1(q_e)+ \xi_1(q_o) = \xi_1(q_e) + \xi_1(q_o - 2q_o\chi_-)\\
        &= \eta(q_e) + \eta(q_o-2q_o\chi_-) \\
        &=  \xi_2(q_e) + \xi_2(q_o - 2q_o\chi_-)= \xi_2(q_e) + \xi_2(q_o) = \xi_2(q).
\end{align*}
Note that $q_o-2q_o\chi_-=q_o(\chi_+-\chi_-)\in \widetilde{QC}$ and that $\lim\limits_{t\to 1+0} q_o(t)\chi_-(t)=0$, 
whence $\xi_i(q_o\chi_-)=0$. It follows that $\xi_1=\xi_2$. 
\end{proof}

\begin{thm}\label{thm32}
Let $\eta\in M_{\pm 1}(\widetilde{QC})$. Then either 
\begin{enumerate}
\item[(a)] 
$M^\eta(QC) = \{\xi\}$ with $\xi=\xi' \in M_{\pm 1}^0(QC)$, or
\item[(b)] 
$M^\eta(QC) = \{\xi, \xi'\}$ with $\xi\in M_{\pm 1}^+(QC)\setminus M_{\pm1}^-(QC)$ and $\xi'\in M_{\pm 1}^-(QC)\setminus M_{\pm 1}^+(QC)$.
\end{enumerate}
\end{thm}
\begin{proof}
{}From the statement  \eqref{eq1.7} it follows that   $\hat{\xi}\in M_{\pm 1}(\tQC)$ implies $\xi\in M_{\pm1}(QC)$.
Therefore $\emptyset \neq M^\eta(QC)\subseteq M_{\pm 1}(QC)$ whenever
$\eta\in M_{\pm1}(\tQC)$. Now the assertion follows from Proposition \ref{prop52}, Proposition \ref{p3.1}, and  the statements (i)-(iii) above.
\end{proof}

Next we want to characterize of  those $\eta \in M_{\pm 1}(\tQC)$ which give rise to the first case.
Consider the functionals $\delta_{\lambda, \tau}\in \widetilde{QC}^*$ associated with the moving average  \eqref{deltam}, and define,
in analogy to \eqref{Mtau0},
\begin{equation}\label{eqn-6.10}
        M_\tau^0(\widetilde{QC}) :=   M(\widetilde{QC}) \cap \clos_{\,\widetilde{QC}^*}(\Lambda \times \{\tau\}).
\end{equation}
We will use this definition for $\tau \in \overline{\T_+}=\T_+\cup\{+1,-1\}$.

\begin{thm}\label{t3.3}
The map $\Psi:\xi\mapsto \hat{\xi}$ is a bijection from  $M_{\pm 1}^0(QC)$ onto $M_{\pm 1}^0(\tQC)$.
\end{thm}
\begin{proof}
Without loss of generality consider the case $\tau=1$.
First of all, $\Psi$ maps $M_1^0(QC)$ into $M_1^0(\widetilde{QC})$.
Indeed, it follows from \eqref{Mtau0} that for any $\xi\in M_1^0(QC)$, any 
$q_1,\dots,q_k\in\tQC\subset QC$ and $\varepsilon>0$, there exists $\lambda\in\Lambda$
such that $|\xi(q_i)-\delta_{\lambda,1}(q_i)|<\eps$ for all $i$. 
But this is just $|\hat{\xi}(q_i)-\delta_{\lambda,1}(q_i)|<\eps$.
Therefore, $\hat{\xi}$ lies in the weak-$^*$ closure of $\{\delta_{\lambda,1}|_{QC^*}\}_{\lambda\in\Lambda}$.
Hence, by \eqref{eqn-6.10}, $\hat{\xi}\in M_{1}^0(\tQC)$.
The injectiveness of the map $\Psi|_{M_1^0(QC)}$ follows from Theorem \ref{thm32} or Proposition \ref{p3.1}.

It remains to show that $\Psi|_{M_1^0(QC)}$ is surjective. Choose any $\eta\in M_1^0(\widetilde{QC})$. 
By definition, there exists a net 
$\{\lambda_\om\}_{\om\in\Omega}$, $\lambda_\om\in \Lambda$, such that the net 
$\{\delta_{\lambda_\om}\}_{\om\in\Omega}:=\{\delta_{\lambda_\om,1}\}_{\om\in\Omega}$ 
converges to $\eta$ (in the weak-$^*$ sense of  functionals on $\tQC$). 
Note that $\delta_\lambda(q) = 0$ for any $\lambda\in \Lambda$ whenever $q\in QC$ is an odd function. 
Therefore the net $\{\delta_{\lambda_\om}\}_{\om\in \Omega}$ (regarded as functionals on $QC$) 
converges to the functional $\xi\in QC^*$ defined by
$$
  \xi(q) := \eta(\frac{q+\tilde{q}}{2}), \quad  q\in QC.
$$
Indeed, $\delta_{\lambda_\om}(q)=\frac{1}{2}\delta_{\lambda_\om}(q+\tilde{q})\to \frac{1}{2}\eta(q+\tilde{q})=\xi(q)$.
It follows that $\xi\in \clos_{\, QC^*}(\Lambda\times\{1\})$.

Next we show that $\xi$ is multiplicative over $QC$, i.e., $\xi\in M(QC)$.
Given arbitrary $p, q\in QC$ we can decompose them into even and odd parts as $p = p_e+p_o$, $q = q_e+ q_o$.
The even part of $pq$ equals $p_eq_e+p_oq_o$. Therefore using the definition of $\xi$ in terms of $\eta$
we get 
$$
\xi(p)\xi(q)=\eta(p_e)\eta(q_e)=\eta(p_eq_e),\quad \xi(pq)=\eta(p_eq_e+p_oq_o).
$$
Hence the multiplicativity of $\xi$ follows if we can show that  $\eta(p_oq_o)=0$. To see this we argue as follows.
By Lemma \ref{cor54}(ac),
we have $p_oq_o|_{M_{\pm1}^0(QC)} = 0$ and $p_oq_o\chi_+\in QC$, and hence by  Lemma \ref{lem53}
the integral gap
$$
  \gamma_1(p_oq_o\chi_+) = \limsup_{\delta\to +0} \left|\frac{1}{\delta}\int_0^\delta(p_oq_o)(e^{ix})\,dx\right| = 0.
$$
In other word, as $\lambda\to+\infty$,
$$
\delta_\lambda(p_oq_o) = \frac{\lambda}{2\pi}\int_{-\pi/\lambda}^{\pi/\lambda} (p_oq_o)(e^{ix})\,dx 
= \frac{\lambda}{\pi}\int_0^{\pi/\lambda}(p_oq_o)(e^{ix})\,dx  \to 0.
$$
Since the net  $\{\delta_{\lambda_\om}\}_{\om\in\Omega}$ (regarded as functionals on $\widetilde{QC}$) converges to $\eta\in M(\widetilde{QC})$,
it follows from Lemma \ref{lem57} that $\lambda_\om\to +\infty$. Therefore,
$$
\delta_{\lambda_\om}(p_oq_o)\to 0 \quad\mbox{and}\quad \delta_{\lambda_\om}(p_oq_o)\to\eta(p_oq_o).
$$
We obtain $\eta(p_oq_o)=0$ and conclude that $\xi$ is multiplicative. Combined with the above this yields $\xi\in M_1^0(QC)$, 
while clearly $\eta=\hat{\xi}$. Hence $\Psi:M^0_{1}(QC)\to M^0_1(\tQC)$ is surjective.
\end{proof}

The previous two theorems imply the following.

\begin{cor}\label{c3.5}
$M_{\pm 1}^0(\tQC)$ is a closed subset of $M_{\pm 1}(\tQC)$. Moreover, 
\begin{enumerate}
\item[(a)] 
if $\eta\in M_{\pm 1}^0(\tQC)$, then $M^\eta(QC) = \{\xi\}$ with  $\xi=\xi' \in M_{\pm 1}^0(QC)$;
\item[(b)] 
if $\eta\in M_{\pm 1 }(\tQC)\setminus M_{\pm 1}^0(\tQC)$, then 
$M^\eta(QC) = \{\xi, \xi'\}$ with $\xi\in M_{\pm 1}^+(QC)\setminus M_{\pm 1}^-(QC)$ and $\xi'\in M_{\pm 1}^-(QC)\setminus M_{\pm 1}^+(QC)$.
\end{enumerate}
\end{cor}

Note also that  $M_{\pm1}(\widetilde{QC})$ decomposes into the disjoint union of 
\begin{equation}\label{f.35}
 M_{\pm1}(\widetilde{QC}) \setminus M_{\pm1}^0(\widetilde{QC}) \quad\mbox{ and }\quad
 M_{\pm1}^0(\widetilde{QC}),
\end{equation}
and that $\Psi$ is a  {\em two-to-one} map from $M_{\pm 1}(QC)\setminus M_{\pm 1}^0(QC)$ onto $M_{\pm 1}(\tQC)\setminus M_{\pm 1}^0(\tQC)$.


\subsection{Fibers over $\boldsymbol{M_\tau(\tQC)}$, $\boldsymbol{\tau\in \T_+}$}

Now we consider the fibers of $M^\eta(QC)$ over $\eta\in M_\tau(\tQC)$ with $\tau\in\T_+$.
This case is easier than the previous one.

\begin{prop}\label{prop59} If $\hat{\xi}_1 = \hat{\xi}_2$ for $\xi_1, \xi_2\in M_\tau(QC)$ with $\tau\in\T_+$, then $\xi_1 = \xi_2$.
\end{prop}
\begin{proof}
Otherwise, there exists a $q\in QC$, such that $\xi_1(q)\neq 0$, $\xi_2(q) = 0$. Since $\tau \in  \T_+$, one can choose a smooth function $c_\tau$ such that $c_\tau = 1$ in a neighborhood of $\tau$ and such that  it vanishes on the lower semi-circle. Now, construct 
$\overline{q} = qc_\tau + \widetilde{qc_\tau}\in\tQC$. Note that  $\overline{q} - q$ is continuous at $\tau$ and vanishes there, hence 
$\xi_1(\overline{q} - q) = \xi_2(\overline{q} - q) = 0$. 
But then, since $\overline{q}\in\tQC$ and $\hat{\xi}_1 = \hat{\xi}_2$, we have
$$
0\neq \xi_1(q) = \xi_1(\overline{q}) = \xi_2(\overline{q}) = \xi_2(q) = 0,
$$
which is a contradiction.
\end{proof}

It has been stated in \eqref{eq1.7} that $\Psi$ maps $M_\tau(QC)\cup M_{\bar{\tau}}(QC)$ onto $M_\tau(\tQC)$.
Taking the statements (i)-(iii) into account, the previous proposition implies the following.

\begin{cor}\label{c3.7}
Let $\tau\in \T_+$ and $\eta\in M_\tau(\tQC)$. Then
$M^\eta(QC)=\{\xi,\xi'\}$
with some (unique) $\xi\in M_\tau(QC)$. 
\end{cor}

This corollary implies that $\Psi$ is a bijection from $M_{\tau}(QC)$ onto $M_\tau(\widetilde{QC})$ for $\tau\in\T_+$. 
Clearly, $\Psi$ is also a bijection from $M_{\bar{\tau}}(QC)$ onto $M_\tau(\widetilde{QC})$.
This suggests to define
\begin{equation}\label{Mpm-tQC}
     M_\tau^\pm(\widetilde{QC}):=\{ \, \hat\xi \,:\, \xi  \in M_\tau^\pm (QC)\,\},\qquad \tau\in \T_+ \,.
\end{equation}
Recall that we defined $M_\tau^0(\tQC)$ by equation \eqref{eqn-6.10}.

\begin{prop}\label{prop510}
For $\tau\in \T_+$ we have
$$
 M_\tau(\widetilde{QC})=M_\tau^+(\widetilde{QC})\cup M_\tau^-(\widetilde{QC}),\qquad
 M_\tau^0(\widetilde{QC})=M_\tau^+(\widetilde{QC})\cap M_\tau^-(\widetilde{QC}).
$$
\end{prop}
\begin{proof}
The first identity is obvious. Regarding the second one, note that by definition and by Proposition \ref{prop59},
$$
   M_\tau^+(\widetilde{QC})\cap M_\tau^-(\widetilde{QC})=\{\,\hat{\xi}\,:\, \xi\in M_\tau^0(QC)\,\}.
$$
It suffices to show that the map $\Psi:M_\tau^0(QC)\to M_\tau^0(\widetilde{QC})$ is well-defined and bijective. 
Similar to the proof of Theorem \ref{t3.3}, it can be shown that it is well-defined. 
Obviously it is injective. It remains to show that it is surjective.

Choose any $\eta\in M_\tau^0(\widetilde{QC})$. By definition, there
exists a net $\{\lambda_\om\}_{\om\in\Omega}$, $\lambda_\om\in\Lambda$, such that the
net  $ \{\delta_{\lambda_\om}\}_{\om\in \Omega}:=\{\delta_{\lambda_\om,\tau}\}_{\lambda_\om\in \Omega}$
converges to $\eta$ (in the weak-$^*$ sense of functionals on $\tQC$). {}From Lemma \ref{lem57} it follows that
$\lambda_\om\to +\infty$.  Choose a  continuous function $c_\tau$ such that $c_\tau = 1$ in a neighborhood of $\tau$ and such that  it vanishes on the lower semi-circle. The net  $\{\delta_{\lambda_\om}\}_{\om\in \Omega}$ (regarded as functionals on $QC$) converges to the functional $\xi\in QC^*$ defined by
$$
    \xi(q) := \eta(\overline{q}),  \quad q\in QC,
$$
where $\overline{q} = qc_\tau + \widetilde{qc_\tau} \in \widetilde{QC}$. 
Indeed, $q-\overline{q}$ vanishes on a neighborhood of $\tau$, and hence
$\delta_\lambda(q)=\delta_\lambda(\overline{q})$ for $\lambda$ sufficiently large. 
Therefore, $\delta_{\lambda_\om}(q)-\delta_{\lambda_\om}(\overline{q})\to 0$.
This together with $\delta_{\lambda_\om}(\overline{q})\to\eta(\overline{q})=\xi(q)$ implies that 
$\delta_{\lambda_\om}(q)\to \xi(q)$. It follows that $\xi\in  \clos_{\, QC^*}(\Lambda\times\{\tau\})$.

In order to show that $\xi$ is multiplicative over $QC$, we write (noting $c_\tau\widetilde{c_\tau}=0$)
$$
   \overline{pq}-\overline{p}\cdot\overline{q}
   =pqc_\tau+\widetilde{pqc_\tau}-(pc_\tau+\widetilde{pc_\tau})(qc_\tau+\widetilde{qc_\tau})
   =pq(c_\tau-c_\tau^2)+\widetilde{pq}(\widetilde{c_\tau}-\widetilde{c_\tau}^2).
$$
This is an even function vanishing in a neighborhood of $\tau$ and $\bar{\tau}$. Therefore $\eta(\overline{pq}-\overline{p}\cdot\overline{q})=0$,
which implies $\xi(pq)=\xi(p)\xi(q)$ by definition of $\xi$. It follows that $\xi\in M(QC)$. Therefore, $\xi \in M_\tau^0(QC)$ by definition \eqref{Mtau0}.
Since $\hat{\xi} = \eta$ this implies surjectivity.
\end{proof}

A consequence of the previous proposition is that $M_\tau(\tQC)$ is the disjoint union of
\begin{equation}\label{f.35n}
  M_\tau^0(\widetilde{QC}),\qquad 
  M_\tau^+(\widetilde{QC}) \setminus  M_\tau^0(\widetilde{QC}),\quad \mbox{ and }\quad 
  M^-_\tau(\widetilde{QC}) \setminus  M^0_\tau(\widetilde{QC}).
\end{equation}
Comparing this with \eqref{f1.8} we obtain that $\Psi$ is a {\em two-to-one} map from 
\begin{itemize}
\item[(i)] 
$M_\tau^+(QC)\setminus M_\tau^0(QC)  \;\cup\;  M_{\bar{\tau}}^-(QC)\setminus M_{\bar{\tau}}^0(QC)$
onto $M_\tau^+(\tQC)\setminus M_\tau^0(\tQC)$,
\item[(ii)] 
$M_\tau^-(QC)\setminus M_\tau^0(QC)  \;\cup\;  M_{\bar{\tau}}^+(QC)\setminus M_{\bar{\tau}}^0(QC)$
onto $M_\tau^-(\tQC)\setminus M_\tau^0(\tQC)$,
\item[(iii)]
$M_\tau^0(QC)  \;\cup\;   M_{\bar{\tau}}^0(QC)$
onto $ M_\tau^0(\tQC)$.
\end{itemize}


\section{Localization of $\boldsymbol{PQC}$ over $\boldsymbol{\tQC}$}
\label{sec:4}

Now we are going to identify the fibers $M^\eta(PQC)$ over $\eta\in \tQC$. This allows us to show that certain
quotient $C^*$-algebras that arise from $PQC$ through localization are isomorphic to concrete $C^*$-algebras. 
What we precisely mean by the latter is the following.

Let $\kA$ be a  commutative $C^*$-algebra and $\kB$ be a $C^*$-subalgebra, both having the same unit element.
For $\beta\in M(\kB)$ consider the smallest closed ideal of $\kA$ containing the ideal $\beta$, 
$$
  \kJ_\beta = \closid_{\,\kA} \{\,  b\in \kB\,:\, \beta(b)=0\,\}.
$$
It is known (see, e.g., \cite[Lemma 3.65]{BS}) that 
$$
  \kJ_\beta = \{ \, a\in\kA \,:\, a|_{M_\beta(\kA)}=0\,\}.
$$
Therein $a$ is identified with its Gelfand transform.
Hence the map
$$
  a+ \kJ_\beta \in \kA/ \kJ_\beta \mapsto a|_{M_\beta(\kA)} \in C({M_\beta(\kA)})
$$
is a well-defined *-isomorphism. In other words,  the quotient algebra $\kA/\kJ_\beta$ is isomorphic to $C(M_\beta(\kA))$.
However, it is often more useful to identify this algebra with a more concrete $C^*$-algebra $\kD_\beta$.
This motivates the following definition. A unital *-homomorphism
$\Phi_\beta: \kA\to \kD_\beta$ is said to  {\em  localize the algebra $\kA$ at $\beta\in M(\kB)$}
if it is surjective and if $\ker \Phi_\beta=\kJ_\beta$. In other words, the induced *-homomorphism
$$
a+ \kJ_\beta \in \kA/ \kJ_\beta \;\; \mapsto \;\;\Phi_\beta(a) \in \kD_\beta
$$
is a *-isomorphism between $\kA/ \kJ_\beta$ and $ \kD_\beta$. 

Our goal is to localize $PQC$ at $\eta\in M(\tQC)$ in the above sense. The corresponding fibers
are 
$$
M^\eta(PQC)=\{ \, z \in M(PQC)\, :\,  z|_{\tQC} =\eta\,\}
=
\{ \, z\in M_\xi(PQC)\,:\, \; \xi\in M^\eta(QC)\,\}.
$$
Hence they can be obtained from the fibers $M^\eta(QC)$ and  $M_\xi(PQC)$ (see Theorem \ref{thm53}).
Recall the identification of  $z\in M(PQC)$ with $(\xi,\sigma)\in M(QC)\times \{+1,-1\}$ given in \eqref{z-xi-si}.
Furthermore, $\C^N$ is considered as a $C^*$-algebra with component-wise operations and maximum norm.
(It is the $N$-fold direct product of the $C^*$-algebra $\C$.)

\begin{thm}\
\begin{enumerate}
\item[(a)]
Let $\eta\in M_{\pm1}^0(\tQC)$ and $M^\eta(QC)=\{\xi\}$. Then $M^\eta(PQC)=\{(\xi,+1),(\xi,-1)\}$ and 
$\Phi: PQC\to \C^2$ defined by 
$$
p\in PC\mapsto  (p(\pm1 +0), p(\pm 1-0)),\qquad 
q\in QC\mapsto (\xi(q),\xi(q))
$$
extends to a localizing *-homomorphism.
\item[(b)]
Let $\eta\in M_{\pm1}(\tQC)\setminus M_{\pm1}^0(\tQC)$ and $M^\eta(QC)=\{\xi,\xi'\}$ with 
$\xi\in M_{\pm 1}^+(QC)\setminus M_{\pm 1}^0(QC)$. 
Then $M^\eta(PQC)=\{(\xi,+1),(\xi',-1)\}$ and 
$\Phi: PQC\to \C^2$ defined by 
$$
p\in PC\mapsto  (p(\pm1 +0), p(\pm 1-0)),\qquad 
q\in QC\mapsto (\xi(q),\xi'(q))
$$
extends to a localizing *-homomorphism.
\item[(c)]
Let $\eta\in M_{\tau}^0(\tQC)$, $\tau\in\T_+$, and $M^\eta(QC)=\{\xi,\xi'\}$ with 
$\xi\in  M_{\tau}^0(QC)$.  Then $M^\eta(PQC)=\{(\xi,+1),(\xi,-1),(\xi',+1),(\xi',-1)\}$ and 
$\Phi: PQC\to \C^4$ defined by 
$$
p\in PC\mapsto  (p(\tau +0), p(\tau-0),p(\bar{\tau} +0), p(\bar{\tau}-0)),\quad 
q\in QC\mapsto (\xi(q),\xi(q),\xi'(q),\xi'(q))
$$
extends to a localizing *-homomorphism.
\item[(d)]
Let $\eta\in M_{\tau}^\pm(\tQC)\setminus M_{\tau}^0(\tQC)$, $\tau\in\T_+$, and $M^\eta(QC)=\{\xi,\xi'\}$ with 
$\xi\in  M_{\tau}^\pm (QC)\setminus M_{\tau}^0(QC)$.  Then $M^\eta(PQC)=\{(\xi,\pm1),(\xi',\mp 1)\}$ and 
$\Phi: PQC\to \C^2$ defined by 
$$
p\in PC\mapsto  (p(\tau \pm0), p(\bar{\tau}\mp0)),\qquad 
q\in QC\mapsto (\xi(q),\xi'(q))
$$
extends to a localizing *-homomorphism.
\end{enumerate}
\end{thm}
\begin{proof}
Note that all cases of $\eta\in M(\tQC)$ are considered (see \eqref{unions},  \eqref{f.35}, and \eqref{f.35n}). 
The description of $M^\eta(QC)$ follows from Corollaries \ref{c3.5} and \ref{c3.7}.

Let us consider only one case, say case (c). The other cases can be treated analogously. 
We can write
$$
M^\eta(PQC)=\{z\in M_\xi(PQC)\,:\, \xi\in M^\eta(QC)\}.
$$
Hence as $M^\eta(QC)=\{\xi,\xi'\}$ in this case, we get $M^\eta(PQC)=M_\xi(PQC)\cup M_{\xi'}(PQC)$. Now use
Theorem \ref{thm53} to get the correct description of $M^\eta(PQC)$ as a set of four elements $\{z_1,z_2,z_3,z_4\}$.
Identifying $C(M^\eta(PQC))=C(\{z_1,z_2,z_3,z_4\})$ with $\C^4$, the corresponding localizing homorphism is given by 
$$
\Phi: f\in PQC \mapsto (z_1(f),z_2(f),z_3(f),z_4(f))\in\C^4.
$$
Using the identification  of $z$ with $(\xi,\sigma)\in M(QC)\times \{+1,-1\}$ as given in  \eqref{z-xi-si}, 
the above form of the *-homomorphism follows by considering $f=p\in PC$ and $f=q\in QC$.
\end{proof}



\end{document}